\documentclass[11pt]{amsart}

 \usepackage{amsmath,amsthm,amsfonts,amssymb,verbatim}
 \usepackage{url}
 \usepackage{graphicx}
 \usepackage[all]{xy}
\usepackage{tikz}\usetikzlibrary{matrix,arrows}
 \usepackage{wrapfig}
\usepackage{picins}
 \usepackage{pinlabel}
 \usepackage{subfigure}
 \usepackage{xfrac}
 \usepackage{hyperref}
 
\newcommand\rank{{\text{rank}\mkern2mu}}

\def\co{\colon\thinspace}

\newcommand{\into}{\hookrightarrow}

\newcommand{\bZ}{\mathbb{Z}}
\newcommand{\bQ}{\mathbb{Q}}
\newcommand{\bR}{\mathbb{R}}

\newcommand{\sL}{\mathcal{L}}

\newcommand{\cutting}{\backslash\!\!\backslash}

\newcommand{\HFhat}{\widehat{\mathit{HF}}}
\newcommand{\CFD}{\widehat{\mathit{CFD}}}

\newtheorem{theorem}{Theorem}
\newtheorem{corollary}[theorem]{Corollary}
\newtheorem{proposition}[theorem]{Proposition}

\newtheorem*{namedtheorem}{\theoremname}
\newcommand{\theoremname}{testing}

\theoremstyle{definition}
\newtheorem{definition}[theorem]{Definition}

\newtheorem{remark}[theorem]{Remark}

\begin{document}

\title{L-spaces, taut foliations, and graph manifolds}

\author{Jonathan Hanselman}
\email{jh66@math.princeton.edu}
\address {Department of Mathematics, Princeton University, Fine Hall, Washington Road, Princeton,
NJ 08540, USA}

\author{Jacob Rasmussen}
\email{J.Rasmussen@dpmms.cam.ac.uk}
\address {Department of Pure Mathematics and Mathematical Statistics,
Centre for Mathematical Sciences, Wilberforce Road, Cambridge CB3 0WB, UK}

\author{Sarah Dean Rasmussen}
\email{S.Rasmussen@dpmms.cam.ac.uk}
\address {Department of Pure Mathematics and Mathematical Statistics,
Centre for Mathematical Sciences, Wilberforce Road, Cambridge CB3 0WB, UK}

\author{Liam Watson}
\email{liam@math.ubc.ca}
\address {Department of Mathematics, University of British Columbia, 1984 Mathematics Road,
Vancouver, BC, V6T 1Z2, Canada }


\thanks{The first author was partially supported by NSF RTG grant DMS-1148490. The second author was partially supported by EPSRC grant EP/M000648/1. The third author was supported by EPSRC grant EP/M000648/1. The fourth author was partially supported by a Marie Curie Career Integration Grant (HFFUNDGRP).}

\begin{abstract} 
If \(Y\) is a closed orientable graph manifold, we show that \(Y\) admits a coorientable taut foliation if and only if \(Y\) is not an L-space. Combined with previous work of Boyer and Clay, this implies that \(Y\) is an L-space if and only if \(\pi_1(Y)\) is not left-orderable. 
\end{abstract}

\maketitle

\section{Introduction}
An L-space is a rational homology sphere $Y$ with simplest possible Heegaard Floer homology,\footnote{We use Floer homology with coefficients in $\bZ/2\bZ$. Other coefficient systems are discussed at the end of the paper.}
in the sense that $\rank\HFhat(Y)= |H_1(Y;\bZ)|$. Ozsv\'ath and Szab\'o have shown, by an argument analogous to one used by Kronheimer and Mrowka in the monopole setting \cite{KM1997}, that the existence of a $C^2$ coorientable taut foliation ensures that $\rank\HFhat(Y)> |H_1(Y;\bZ)|$ \cite{OSz2004}. That is, L-spaces do not admit  $C^2$ coorientable taut foliations. 

For certain classes of manifolds the converse is known to hold. In particular, for Seifert fibered spaces with base orbifold $S^2$, Lisca and Stipsicz have shown that if $Y$ is not an L-space then $Y$ admits a coorientable taut foliation \cite{LS2007}. (In fact, it can be shown that the two conditions are equivalent for all Seifert fibered spaces; see \cite{BGW2013}.) The main result of this note extends Lisca and Stipsicz's result to general graph manifolds. Recall that a graph manifold is a prime three-manifold admitting a JSJ decomposition into pieces admitting Seifert fibered structures.

\begin{theorem}\label{Main Theorem}Let $Y$ be a closed, connected, orientable graph manifold. If $Y$ is not an L-space then $Y$ admits a $C^0$ coorientable taut foliation.  \end{theorem}

An independent alternative proof of this result, together with an explicit classification of graph manifolds admitting cooriented taut foliations, appears in \cite{R2015}, by the third author.

There is a third condition on three-manifolds that is relevant in this setting. Recall that a countable group is left-orderable if it admits an effective action on $\bR$ by order-preserving homeomorphisms \cite{BRW2005}. There is a conjectured equivalence among prime three-manifolds between L-spaces and non-left-orderability of the fundamental group \cite{BGW2013}. Theorem \ref{Main Theorem} gives rise to an equivalence between all three conditions for graph manifolds.

\begin{theorem}\label{Equivalence Theorem}If $Y$ is a closed, connected, orientable graph manifold then the following are equivalent:
\begin{itemize}
\item[(i)] $Y$ is not an L-space;
\item[(ii)] $Y$ admits a $C^0$ coorientable taut foliation;
\item[(iii)] $Y$ has left-orderable fundamental group.
\end{itemize}
\end{theorem}

The equivalence (ii) $\Leftrightarrow$ (iii) is due to Boyer and Clay \cite{BC}. The implication (ii) \(\Rightarrow\)  (i) is established by Boyer and Clay in \cite{BC-prep}. Alternately, this implication follows from a theorem of Bowden \cite{Bowden} and, independently, Kazez and Roberts \cite{KR, KR-prep} that taut \(C^0\) foliations can be approximated by weakly semi-fillable contact structures, together with the earlier work of  Ozsv\'ath and Szab\'o \cite{OSz2004}. Theorem \ref{Main Theorem} provides the final required implication (i) $\Rightarrow$ (ii).
Among graph manifolds, the above equivalence was known for Seifert fibered spaces (see \cite{BGW2013} and references therein). The case of a graph manifold with a single JSJ torus was shown in \cite[Theorem 1.1]{HW} of the first and fourth authors; this case also follows from the second and third authors' gluing theorem \cite[Theorem 6.2]{RR}. In fact, Theorem \ref{Main Theorem} (and thus Theorem \ref{Equivalence Theorem}) also follows from results in \cite{RR}, as we aim to show in this paper.

The equivalence (i) $\Leftrightarrow$ (iii) resolves \cite[Conjecture 1]{BGW2013} in the affirmative for graph manifolds. We thank Tye Lidman for pointing out the following immediate consequence:

\begin{corollary}
Suppose $f\co Y_1\to Y_2$ is a non-zero degree map between closed, connected, orientable graph manifolds. If $Y_1$ is an L-space then $Y_2$ is an L-space as well. 
\end{corollary}
\begin{proof}This follows from Theorem \ref{Equivalence Theorem} and \cite[Theorem 3.7]{BRW2005}. Note that the existence of the non-zero degree map $f$ induces a non-trivial homomorphism from $\pi_1(Y_1)$ to $\pi_1(Y_2)$ \cite[Lemma 3.8]{BRW2005}. Hence if $\pi_1(Y_2)$ is left-orderable then so is $\pi_1(Y_1)$ \cite[Theorem 1.1]{BRW2005}.\end{proof}

Our work rests on a detailed study of the Heegaard Floer invariants of orientable three-manifolds $M$ with torus boundary. Denote by $M(\alpha)$ the result of Dehn filling along a slope $\alpha$ in $\partial M$, that is, $\alpha$ represents a primitive class in $H_1(\partial M;\bZ)/\{\pm 1\}$. The set of slopes $Sl(M)$ may be identified with the extended rationals $\bQ\cup\{\frac{1}{0}\}$, viewed as a subspace of $\bR \text{P}^1$. Consider the set of  L-space slopes $\sL_M = \{\alpha \, |\,  M(\alpha) \text{\ is\ an\ L-space} \}$; its interior $\sL_M^\circ$ is the set of {\em strict L-space slopes}. The key tool used in the proof of Theorem \ref{Main Theorem} is the following non-L-space cutting theorem, which follows from results proved in~\cite{RR}.

\begin{theorem}\label{Cutting Theorem}
Let $N$ denote the twisted $I$-bundle over the Klein bottle, with rational longitude~$\lambda$. Let $M_1$ and $M_2$ be compact, connected, orientable three-manifolds with torus boundary, and suppose that $Y\cong M_1\cup_h M_2$ for some homeomorphism $h\co \partial M_1\to \partial M_2$. If $Y$ is not an L-space, then
\begin{itemize}
\item[(1)]
there exists a slope $\alpha$ in $\partial M_1$ such that $\alpha \not\in \mathcal{L}^\circ_{M_1}$ and $h(\alpha) \not\in\mathcal{L}^\circ_{M_2}$; moreover, 
\item[(2)]
for any orientation-reversing homeomorphisms 
$\varphi_i : \partial N \to \partial M_i$ with
$\varphi_{1}(\lambda) = \alpha$ and
$\varphi_{2}(\lambda) = h(\alpha)$,
the closed manifolds
$N \cup_{\varphi_1} M_1$ and $N \cup_{\varphi_2} M_2$
are non-L-spaces.
\end{itemize}
\end{theorem}
We note that statement~(1) alternatively results from an enhanced gluing result introduced in Theorem~\ref{Gluing Theorem} below, which is of independent interest.

\section{Notions of simplicity} Before proving Theorem \ref{Cutting Theorem} we recall the main notions of \cite{HW} and \cite{RR} in order to highlight a key point of interaction between these two works. 
The subject of \cite{RR} is the class of  {\em Floer simple} manifolds: A manifold with torus boundary \(M\) is Floer simple 
if and only if \(\mathcal{L}_{M}\) contains more than one element \cite[Proposition 1.3]{RR}. 
In \cite{HW}, the main object of study is the class of  {\em simple loop-type} manifolds. The bordered Floer homology \cite{LOT} of these manifolds has a particularly nice form. Below, we briefly summarize some relevant facts about bordered Floer homology. For a more detailed exposition we refer to \cite[Section 2]{HW}. 

The bordered Floer module \(\CFD\)  is an invariant of a three-manifold with parametrized boundary. When \(\partial M\) is a torus we can specify a parametrization of \(\partial M\) by a pair of simple closed curves \(\alpha, \beta \in H_1(\partial M;\bZ)\) with \(\alpha \cdot \beta = 1\). In this case, the bordered Floer homology \(\CFD(M, \alpha, \beta)\) may be represented by a directed graph whose edges are labeled by elements of the set \(\mathcal{A} = \{\rho_1,\rho_2,\rho_3, \rho_{12}, \rho_{23}, \rho_{123}\}\).  The triple \((M,\alpha, \beta)\) is said to be of {\em  loop-type} if each vertex in the graph representing \(\CFD(M,\alpha, \beta)\) has valence \(2\) \cite[Definition 3.2]{HW}. Such a graph can be decomposed into certain standard {\em puzzle pieces} as described in \cite[Section~3]{HW}. For our  purposes, the  relevant pieces are the ones shown in Figure~\ref{Fig:Unstable Puzzle Pieces}. The property of being loop-type is inherent to the underlying manifold $M$: If the triple $(M, \alpha, \beta)$ is of loop-type for some choice of parametrizing curves $\alpha$ and $\beta$, then it is of loop-type for any choice of $\alpha$ and $\beta$. In this case, we say the manifold $M$ is of loop-type.

\begin{figure}[h]
\centering
\labellist
\pinlabel {\begin{tikzpicture}[>= latex, scale=0.7]
 \node at (1.5,0) {$\circ$};  \node at (3,0) {$\circ$}; \node at (4.5,0) {$\bullet$};
\draw[->, thick, shorten >=0.1cm,shorten <=0.3cm]  (0,0) to (1.5,0);
\draw[<-, thick, shorten >=0.1cm,shorten <=0.1cm, dashed]  (1.5,0) to (3,0);
\draw[<-, thick, shorten >=0.1cm,shorten <=0.1cm] (3,0) to (4.5,0);
\node at (0.75,0.2) {$\scriptstyle\rho_1$};\node at (2.25,0.2) {$\scriptstyle\rho_{23}$};\node at (3.75,0.2) {$\scriptstyle\rho_3$};
\end{tikzpicture}} at 79.5 32  
\pinlabel {$\underbrace{\phantom{aaaaaa}}_k$} at 76 18
\endlabellist
\includegraphics[scale=0.7]{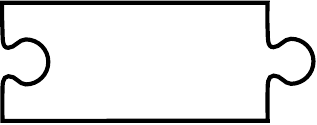}
\qquad
\labellist
\pinlabel {\begin{tikzpicture}[>= latex, scale=0.7]
\node at (1.5,0) {$\circ$};  \node at (3,0) {$\circ$};  \node at (4.5,0) {$\bullet$};
\draw[->, thick, shorten >=0.1cm,shorten <=0.3cm]  (0,0) to (1.5,0);
\draw[->, thick, shorten >=0.1cm,shorten <=0.1cm, dashed]  (1.5,0) to (3,0);
\draw[->, thick, shorten >=0.1cm,shorten <=0.1cm] (3,0) to (4.5,0);
\node at (0.85,0.2) {$\scriptstyle\rho_{123}$};\node at (2.2,0.2) {$\scriptstyle\rho_{23}$};\node at (3.7,0.2) {$\scriptstyle\rho_2$};
\end{tikzpicture}} at 79.5 32  
\pinlabel {$\underbrace{\phantom{aaaaaa}}_k$} at 76 18
\endlabellist
\includegraphics[scale=0.7]{in-out} 
\qquad
\labellist
\pinlabel {\begin{tikzpicture}[>= latex, scale=0.7]
 \node at (4.5,0) {$\bullet$};
\draw[->, thick, shorten >=0.1cm,shorten <=0.3cm]  (0,0) to (4.5,0);
\node at (2.25,0.2) {$\scriptstyle\rho_{12}$};
\end{tikzpicture}} at 79.5 32  
\endlabellist
\includegraphics[scale=0.7]{in-out}
\caption{ Puzzle pieces for simple loops; see \cite{HW}. The pieces represented by the letters $\bar c_k$, $d_k$ and $e$ are shown from left to right where $k$ is a positive integer determining the number of $\circ$ vertices. A simple loop may be expressed as a cyclic word in these letters. }\label{Fig:Unstable Puzzle Pieces}
\end{figure}

\begin{definition} \cite[Definition 4.19]{HW} A manifold with torus boundary $M$ is of simple loop-type if it is of loop-type, the number of connected components of the graph is equal to the number of \(spin^c\) structures on \(M\), and for some choice of parametrizing curves $\alpha$ and $\beta$, $\CFD(M, \alpha, \beta)$ is expressible in the letters $\bar{c}_k$, $d_k$, and $e$.
\end{definition}

\begin{proposition}\label{Simplicity Proposition} \(M\) is Floer simple if and only if \(M\) is of simple loop-type. 
\end{proposition}

\begin{proof} The bordered Floer homology of a Floer simple manifold \(M\) was explicitly computed in \cite[Proposition 3.9]{RR} for an appropriate choice of parametrization \((\alpha, \beta)\). In the course of the proof, it is shown that \(\CFD(M,\alpha,\beta)\) is composed of puzzle pieces of type \(\bar c_k\). Thus to see that \(M\) is of simple loop type, we need only check that the number of loops is equal to the number of \(spin^c\) structures on \(M\), which is \(|H^2(M;\bZ)| = |H_1(M,\partial M;\bZ)| = |H_1(M;\bZ)/\langle \alpha, \beta \rangle| \). 

 Each vertex \(v\) of \(\CFD(M, \alpha, \beta)\) is labeled by a relative \(spin^c\) structure \(\mathfrak{s}(v)\), which we can view as an element of \(H^2(M, \partial M;\bZ) \cong H_1(M;\bZ)\). By \cite[Lemma 3.8]{RR}, edges of the graph labeled by \(\rho_1\) preserve this labeling, edges labeled by \(\rho_{23}\) shift the labeling by \(\alpha \) and edges labeled by \(\rho_{3}\) shift the labeling by \(\alpha + \beta\). 

Given a puzzle piece in \(\CFD(M, \alpha, \beta)\), let \(v\) be its unique black vertex, and label the piece by the image of \(\mathfrak{s}(v)\) in \(H_1(M;\bZ)/\langle \alpha \rangle\). This labeling defines a bijection between the set of puzzle pieces in \(\CFD(M, \alpha, \beta)\) and  \(H_1(M;\bZ)/\langle \alpha \rangle\). Moreover, if the label on a given piece is \(a\), the label on the next piece in the loop is \(a + \beta\). It follows that the set of loops is in bijection with \((H_1(M;\bZ)/\langle \alpha \rangle)/\langle \beta \rangle  \cong H_1(M;\bZ)/\langle \alpha, \beta \rangle\).
Thus a Floer simple manifold is simple loop-type.

Conversely, given a simple loop-type manifold $M$, fix parametrizing curves $\alpha$ and $\beta$ such that $\CFD(M, \alpha, \beta)$ consists only of segments of type $\bar{c}_k$, $d_k$, and $e$. The slope $\infty$ is a strict L-space slope for $(M, \alpha, \beta)$ by \cite[Proposition 4.18]{HW}; that is $\alpha \in \sL^\circ_M$. This implies that $| \sL_M | > 1$ and therefore, by \cite[Proposition 1.3]{RR}, that $M$ is Floer simple.
\end{proof}

\

\section{Cutting and Gluing}

Combining Proposition \ref{Simplicity Proposition} with the gluing theorem \cite[Theorem 1.3]{HW} for simple loop-type manifolds, we obtain a gluing result for Floer simple manifolds.
\begin{theorem}\label{Gluing Theorem}
Suppose that $M_1$ and $M_2$ are  Floer simple manifolds, and consider the closed manifold $M_1\cup_h M_2$  for some homeomorphism 
$h\co \partial M_1\to\partial M_2$. 
\begin{itemize}
\item[(1)] If neither $M_1$ nor $M_2$ are solid torus-like, then $M_1\cup_h M_2$ is not an L-space if and only if there is a slope $\alpha$ in $\partial M_1$ such that  $\alpha \not\in \mathcal{L}^\circ_{M_1}$ and $h(\alpha) \not\in\mathcal{L}^\circ_{M_2}$.
\item[(2)] If either $M_1$ or $M_2$ is solid-torus like, then $M_1\cup_h M_2$ is not an L-space if and only if there is a slope $\alpha$ in $\partial M_1$ such that  $\alpha \not\in \mathcal{L}_{M_1}$ and $h(\alpha) \not\in\mathcal{L}_{M_2}$. \end{itemize}
\end{theorem}
The two cases arising in this statement are expected: the second accounts for Dehn filling (that is, when one of the $M_i$ is a solid torus) and simply verifies the definition of an L-space slope. More generally, we must appeal to a larger class of manifolds which are called \emph{solid torus-like} \cite[Definition 3.23]{HW}, as they are characterized by having bordered Floer homology which resembles that of a solid torus in every $spin^c$ structure \cite{HW}, or equivalently, by having empty $\mathcal{D}^{\tau}$ in the sense of \cite{RR}.   It was proved in \cite{Gillespie} that a solid torus-like manifold must be a solid torus connected sum with an L-space. In particular, if we assume that $M_1$ and $M_2$ are boundary incompressible, then the conclusion in case $(1)$ holds.

{\bf The proof of Theorem \ref{Cutting Theorem}.} According to 
\cite[Theorem 1.6]{RR}, the set $\sL_{M_i}^\circ$ (for $i=1,2$) is either empty or it is (the restriction  to $\bQ\cup\{\frac{1}{0}\}$ of) a connected interval with rational endpoints. In the case $\mathcal{L}^\circ_{M_2} = \emptyset$, let $\alpha$ be the rational longitude of $M_1$, which is not an L-space slope. Similarly, if $\mathcal{L}^\circ_{M_1} = \emptyset$ we choose $\alpha$ such that $h(\alpha)$ is the rational longitude of $M_2$. If $\mathcal{L}^\circ_{M_1}$ and $\mathcal{L}^\circ_{M_2}$ are both nonempty, then $M_1$ and $M_2$ are both Floer simple \cite[Proposition 1.3]{RR}; since $Y$ is not an L-space it follows from Theorem \ref{Gluing Theorem} that there is a slope $\alpha\not\in \mathcal{L}^\circ_{M_1} \subset \mathcal{L}_{M_1}$ such that $h(\alpha) \not\in\mathcal{L}^\circ_{M_2} \subset\mathcal{L}_{M_2}$ as required.

Part (2) of Theorem~\ref{Cutting Theorem} is subsumed as a special case of \cite[Proposition 7.9]{RR} by the second and third authors, but the result still merits some explanation.
Again, and henceforth in this paper, $N$ denotes the twisted $I$-bundle over the Klein bottle, with rational longitude~$\lambda$.
It is straightforward to compute, e.g. from \cite[Theorem 5.1]{RR}, that $\mathcal{L}_N = \mathcal{L}_N^{\circ} = Sl(N) \setminus \{\lambda\}$.  Thus, for $M$ Floer simple,
Theorem \ref{Gluing Theorem} implies that
for any gluing map $\varphi \co \partial N \to \partial M$,
\begin{equation*}
N \cup_{\varphi}\mkern-2.5mu M \,\text{ is a non-L-space}
\;\iff\;
\varphi(\lambda) \notin \mathcal{L}_{M}^{\circ}.
\end{equation*}
Note that, similar to a Dehn filling, the above non-L-space criterion for $N \cup_{\varphi}\mkern-2.5mu M$ depends only on the slope $\varphi(\lambda) \in Sl(M)$, and is independent of the choice of framing, relative to~$\lambda$, of $\varphi$.  In fact, even the $\mathbb{Z}/2\mathbb{Z}$-graded groups $\widehat{HF}(N \cup_{\varphi}\mkern-2.5mu M)$ are independent of this choice of framing \cite[Proposition~21]{BGW2013}, although we will not need this stronger statement.  We therefore call any $N \cup_{\varphi}\mkern-2.5mu M$ with $\alpha = \varphi(\lambda)$ an $N${\textit{-filling}} of $M$ along the slope~$\alpha$.
The above non-L-space criterion---that an $N$-filling is a non-L-space if and only the filling is along a non strict-L-space slope---also holds for $N$-fillings of an arbitrary rational homology solid torus $M$, but the argument is more subtle in the non-Floer-simple case.
\hfill$\Box$

\section{The main theorem}
 
{\bf Decomposing along tori.} 
Since Theorem \ref{Cutting Theorem} generically produces lower-complexity closed non-L-spaces from a closed non-L-space, it provides an {\em{iterative}} decomposition tool for the proof of Theorem~\ref{Main Theorem}.  Moreover, at each decomposition step, the new $N$-filling slopes record one non strict-L-space slope for each boundary component, with these slopes pairwise identified under the original gluing maps.  
In particular, by iteratively decomposing a non-L-space graph manifold along a suitable JSJ decomposition, we can produce a collection of non-L-space $N$-filled Seifert fibered spaces, 
with $N$-filling slopes specifying gluing-compatible non strict-L-space slopes on all boundary components of the Seifert fibered spaces.
Appealing to \cite[Proposition 7.9]{RR} translates these data into the language of NLS-detected slopes in the sense of Boyer and Clay \cite[Definition 7.16]{BC}, whose machinery then automatically produces a cooriented taut foliaton on the original graph manifold.

To describe this process in more detail, we first need to set up some notation for cutting along tori. Note that we may assume without loss of generality that $b_1(Y)=0$: in the case that $b_1(Y)>0$ work of Gabai guarantees the existence of a coorientable taut foliation \cite{Gabai1983}. Thus, every torus $T\into Y$ separates $Y$ into two rational homology solid tori.

Given $Y$ and a collection of disjoint embedded tori $T_1, \ldots, T_n$ in $Y$, let $Y \cutting \{T_i\}$ denote the result of cutting $Y$ along each torus $T_i$. Since every torus is separating, this process produces $n+1$ pieces; that is, $Y \cutting \{T_i\} \cong M_1 \amalg \ldots \amalg M_{n+1}$, where each $M_j$ is a three-manifold with $\partial M_j$ a disjoint union of tori.

If we further specify a collection of slopes $\alpha_* = (\alpha_1, \ldots, \alpha_n)$ on each of the tori in $\{T_i\}$, we can extend each $M_j$ (for $1\le j \le n+1$) to a closed manifold $Y_j^{\alpha_*}$ in the following way: The collection of slopes $\alpha_*$ induces a collection of slopes on the boundary tori of each $M_j$, since each boundary component of each $M_j$ is identified with one of the $T_i$. A closed manifold $Y_j^{\alpha_*}$ is obtained from $M_j$ by gluing a copy of \(N\) (the twisted $I$-bundle over the Klein bottle) to each boundary of $M_j$ such that, for each gluing, the slope $\lambda$ in $\partial N$ is identified with the slope in the relevant component of $\partial M_j$ induced by $\alpha_*$. We say that $Y_j^{\alpha_*}$ is an \emph{N-filling} of $M_j$ along the slopes induced by $\alpha_*$.

Note that the manifold $Y_j^{\alpha_*}$ described above is not uniquely determined by $M_j$ and $\alpha_*$ since each time a copy of $N$ is glued to $M_j$ there is an infinite family of gluing maps which take $\lambda$ in $\partial N$ to the desired slope in $\partial M_j$. A particular gluing map is specified by choosing slopes dual to $\lambda$ in $\partial M$ and dual to each slope induced by $\alpha_*$ in $\partial M_j$; the manifold $Y_j^{\alpha_*}$ depends on the particular choice of dual slopes.  However, Theorem~\ref{Cutting Theorem} again tells us that the question of whether $Y_j^{\alpha_*}$ is an L-space is determined solely by $\alpha_*$, thus is independent of these choices.  Incidentally, one again has the stronger result that the $\mathbb{Z}/2\mathbb{Z}$-graded Heegaard Floer groups $\widehat{HF}(Y_j^{\alpha_*})$ are independent of these choices \cite[Proposition~21]{BGW2013}.

{\bf Slope detection.} Given a three-manifold $Y$, a collection of tori $\{T_i\}_{i=1}^n$ and a collection of slopes $\alpha_*$, we have explained how to construct manifolds $\{Y_j^{\alpha_*}\}_{j=1}^{n+1}$, and observed that the $\mathbb{Z}/2\mathbb{Z}$-graded groups \(\HFhat(Y_j^{\alpha_*})\) do not depend on the choices made in the construction. In particular, whether or not each  $Y_j^{\alpha_*}$ is an L-space is a well--defined question. The key step in proving Theorem \ref{Main Theorem} is the following.

\begin{proposition}\label{main proposition}
Let $Y$ be an irreducible three-manifold and fix a collection of disjoint embedded tori $\{T_1, \ldots, T_n\}$ in $Y$ such that each torus is separating. If $Y$ is a non-L-space, then there is some collection of slopes $\alpha_*$ on these tori with the property that each of the manifolds $Y_j^{\alpha_*}$ defined above is a non-L-space.
\end{proposition}
\begin{proof}
First observe that if $n=1$ (that is, the collection of tori consists of just one torus), this is equivalent to Theorem~\ref{Cutting Theorem}.  In this case, $Y \cong M_1 \cup_h M_2$ for some gluing map $h$. By Theorem~\ref{Cutting Theorem}, there is a slope $\alpha$ in $\partial M_1$ such that $N$-filling $M_1$ along $\alpha$ gives a non-L-space and $N$-filling $M_2$ along $h(\alpha)$ gives a non-L-space. Let $\alpha_1$ be the slope in $T_1$ that corresponds to the slopes $\alpha \in \partial M_1$ and $h(\alpha) \in \partial M_2$.  Then $\alpha_* = (\alpha_1)$ gives the desired collection of slopes.

For the general case, we proceed by induction on $n$. Assume $n>1$ and the result holds for collections of fewer than $n$ tori. First cut $Y$ along the torus $T_1$ to produce two manifolds $M_1$ and $M_2$. By the $n=1$ case, there is a slope $\alpha_1$ in $T_1$ such that $N$-filling $M_1$ and $M_2$ along the slopes corresponding to $\alpha_1$ produces non-L-spaces. We denote the resulting closed manifolds by $Y_1^{(\alpha_1)}$ and $Y_2^{(\alpha_1)}$.

Having cut along $T_1$, the remaining collection of tori $\{T_2, \ldots, T_n\}$ splits into two subsets depending on whether each torus is contained in $M_1$ or $M_2$. Up to relabeling the tori, we may assume that $\{T_2, \ldots, T_m\}$ is the subset of tori contained in in $M_1$ and $\{T_{m+1}, \ldots, T_n\}$ is the subset of tori contained in $M_2$, for some $1\le m\le n$ (note that if $m=1$ the first subset is empty, and if $m=n$ the second subset is empty). We consider these subsets as collections of tori on $Y_1^{(\alpha_1)}$ and $Y_2^{(\alpha_1)}$, respectively. Note that each collection has at most $n-1$ tori.

By the inductive hypothesis applied to $Y_1^{(\alpha_1)}$ with the collection of tori $\{T_2, \ldots, T_m\}$, there is a collection of slopes $(\alpha_2, \ldots, \alpha_m)$ such that cutting along each torus and $N$-filling along the corresponding slopes produces only non-L-spaces. Similarly, there is collection of slopes $(\alpha_{m+1}, \ldots, \alpha_n)$ on the tori $\{T_{m+1}, \ldots, T_n\}$ in $Y_2^{(\alpha_1)}$ with this property. Finally, observe that the non-L-space manifolds obtained from $Y_1^{(\alpha_1)}$ by this process together with the non-L-space manifolds obtained from $Y_2^{(\alpha_1)}$ are exactly the same as the manifolds obtained by cutting $Y$ along the tori $\{T_1, \ldots, T_n\}$ and $N$-filling along the slopes induced by $\alpha_* := (\alpha_1, \ldots, \alpha_n)$.
\end{proof}

The proposition above can be restated using the notion of non-L-space (NLS) detected slopes defined in \cite{BC}. Let $M$ be a manifold with $\partial M$ a disjoint union of $n$ tori, and let $\alpha_* = (\alpha_1, \ldots, \alpha_n)$ be a collection of slopes on the boundary tori. Following \cite[Definition 7.2]{BC}, let $\mathcal{M}_t(\emptyset;[\alpha_*])$ denote the collection of manifolds obtained by filling each boundary component of $M$ by a copy of $N_t$ where the rational longitude of the $i^{\text{th}}$ copy of $N_t$ is sent to $\alpha_i$. The manifold $N_t$ is the Seifert fibered space over the disk with two cone points of order $t$ and Seifert invariants $(\frac{1}{t},\frac{t-1}{t})$. Note that $N_2=N$. In particular, in the notation introduced above, the set $\mathcal{M}_2(\emptyset;[\alpha_*])$ is the set of all possible $N$-fillings $Y^{\alpha_*}$ of $M$ along the slopes $\alpha_*$. Recall that all manifolds in this set have the same Heegaard Floer homology. According to \cite[Definition 7.16]{BC}, the collection of slopes $\alpha_*=(\alpha_1,\alpha_2,\ldots,\alpha_n)$ is non-L-space detected (or NLS detected) if for every $t>1$, the set $\mathcal{M}_t(\emptyset;[\alpha_*])$ contains no L-spaces.
We can now restate Proposition \ref{main proposition} as follows:

\begin{proposition}\label{prop:collection of NLS slopes}
Let $Y$ be an irreducible three-manifold with $b_1 = 0$ and fix a collection of disjoint tori $\{T_1, \ldots, T_n\}$ in $Y$. If $Y$ is a non-L-space, then there is some collection of slopes $\alpha_*$ on these tori with the property that the restriction of $\alpha_*$ to each $M_j$ in $Y \cutting \{T_i\}$ is NLS detected.
\end{proposition}
\begin{proof}
In \cite[Proposition 7.9]{RR}, the second and third authors show that for any rational homology solid torus $M$, ``generalized solid torus'' $N'$
with rational longitude $\lambda'$,
and gluing map $\varphi : \partial N' \to \partial M$, the closed manifold $N' \cup_{\varphi}\mkern-2.5mu M$ is a non-L-space
if and only if 
$\varphi(\lambda') \notin \mathcal{L}_{M}^{\circ}$.
In the discussion preceding that proposition, they also prove that $N_t$ is a generalized solid torus for any $t>1$.
Thus, for a manifold with slope $\alpha_*$, the set
$\mathcal{M}_2(\emptyset, [\alpha_*])$ contains a non-L-space
if and only if for all $t > 1$, the set 
$\mathcal{M}_t(\emptyset, [\alpha_*])$ contains no L-spaces,
hence if and only if $\alpha_*$ is NLS detected.

By Proposition \ref{main proposition}, there is a collection of slopes $\alpha_*$ such that each $Y_j^{\alpha_*}$ is a non-L-space. For each $M_j$, where $1\le j \le n+1$, $Y_j^{\alpha_*}$ is by construction a non-L-space element of $\mathcal{M}_2( \emptyset; [\alpha_*^j] )$, where $\alpha_*^j$ denotes the restriction of $\alpha_*$ to $\partial M_j$.
Thus each $\alpha_*^j$ is NLS detected on $M_j$.
\end{proof}

\begin{remark}Notice that we have yet to restrict to graph manifolds, or even to incompressible tori. Indeed, given a rational homology sphere $Y$ that is not an L-space, a collection of disjoint tori $\{T_i\}$ always gives rise to an NLS detected collection of slopes on the boundary of each component of $Y\cutting \{T_i\}$. This suggests that the same behaviour for taut foliations and/or for left-orders on the fundamental group should be explored
for more general prime three-manifolds.
\end{remark}

{\bf The proof of Theorem \ref{Main Theorem}.}
When $Y$ is a graph manifold and $\{T_i\}$ is the collection of JSJ tori, note that Proposition \ref{prop:collection of NLS slopes} verifies one direction of Boyer and Clay's Conjecture 1.10 in \cite{BC} about cutting and gluing along NLS-detected slopes, namely that a non-L-space graph manifold can be cut into Seifert fibered pieces with gluing-compatible NLS-detected slopes on all boundary components.  This allows us to complete the proof of Theorem \ref{Main Theorem}.

Suppose that $Y$ is a non-L-space graph manifold with $b_1(Y) = 0$. If $Y$ has a trivial JSJ decomposition, then $Y$ is a Seifert fibered and therefore already known to admit a cooriented taut foliation.  Next suppose $Y$ has a non-trivial JSJ decomposition. Since $b_1(Y) = 0$, every JSJ torus separates; we take $\{T_i\}$ to be JSJ tori such that the components of $Y\cutting\{T_i\}$ are Seifert fibered. By Proposition \ref{prop:collection of NLS slopes} there is a collection of slopes $\alpha_* = (\alpha_1, \ldots, \alpha_n)$, with $\alpha_i \in T_i$, such that the restrictions of $\alpha_*$ to each component of $Y\cutting\{T_i\}$ are NLS detected. In \cite[Theorem 8.1]{BC}, Boyer and Clay show that on Seifert fibered spaces, NLS detected slopes are equivalent to what they call ``foliation detected'' slopes.  Thus, we have finally produced foliation-detected slopes on all the JSJ tori decomposing $Y$ into Seifert fibered pieces, and that is precisely what Boyer and Clay's foliation gluing theorem \cite[Theorem 1.7]{BC}
requires, in order to guarantee the existence of a cooriented 
taut foliation on $Y$.
\hfill$\Box$

{\bf Coefficients.}
Up until this point, we have used Floer homology with coefficients in \(\bZ/2\bZ\). This choice was imposed  by our use of bordered Floer homology, which is only defined over \(\bZ/2\bZ\). We now briefly discuss what happens for other coefficient systems. If \(G\) is an abelian group, we say \(Y\) is a \(G\) L-space if 
\(\HFhat(Y,\mathfrak{s}; G) \cong G\) for all \(\mathfrak{s} \in spin^c(Y)\). For a closed orientable graph manifold \(Y\), we consider the following conditions: 

\begin{enumerate}
\item \(Y\) is a \(\bZ\) L-space. 
\item \(Y\) is a \(\bZ/2\bZ\) L-space.
\item \(Y\) does not admit a \(C^0\) coorientable taut foliation. 
\end{enumerate}

In fact, if \(Y\) is a closed graph manifold, these three conditions are equivalent. To see this, we briefly sketch the points at which our argument used \(\bZ/2\bZ\) coefficients. First, the proof of our enhanced gluing result Theorem \ref{Gluing Theorem} depends on bordered Floer homology, hence requires
\(\bZ/2\bZ\) coefficients.
For the proof of Theorem~\ref{Cutting Theorem}, however, 
Theorem \ref{Gluing Theorem} can be replaced with
\cite[Theorem 1.1]{RR}, which works over $\mathbb{Z}$ coefficients. The theorem of Lisca and Stipsicz uses \(\bZ/2\bZ\) coefficients; however for any manifold obtained by Dehn filling a Floer simple manifold, the properties of being a \(\bZ\) L-space and a \(\bZ/p\bZ\) L-space are equivalent \cite[Proof of Proposition 3.6]{RR}. 
Moreover, \cite[Theorem 5.1]{RR} reproves Lisca and Stipsicz's result over $\mathbb{Z}$ by performing a direct computation of the L-space slope interval for any Seifert fibered space over the disc.
Thus, Theorems \ref{Main Theorem}, \ref{Equivalence Theorem}, and
\ref{Cutting Theorem} hold over $\mathbb{Z}$.

{\bf Closing remarks.}
The first author has applied bordered Floer homology to give an algorithm for computing the Heegaard Floer homology of an arbitrary graph manifold \cite{Hanselman2013}. This has been implemented on computer, with considerable savings in computation time if the combinatorics developed in \cite{HW} are incorporated (see \cite[Remark 6.10]{HW} in particular). As a consequence of Theorem \ref{Equivalence Theorem} two questions are now algorithmically decidable:
\begin{itemize}
\item[] Does a graph manifold $Y$ admit a coorientable taut foliation?
\item[] Does a graph manifold $Y$ have a left-orderable fundamental group?
\end{itemize}
The answer to either question is {\em yes} if and only if $\rank\HFhat(Y)>\chi(\HFhat(Y))$ (recall that $\chi(\HFhat(Y))= |H_1(Y;\bZ)|$). This gives a direct, in fact combinatorial, verification of two conditions on a three-manifold that seem quite difficult to certify in general.  

\textbf{Acknowledgements:}{ We thank Jonathan Bowden, Steve Boyer, and Adam Clay  for sharing their preprints \cite{Bowden,BC-prep} with us;  Steve Boyer, Cameron Gordon, Tye Lidman, and the referees for helpful comments on a previous version of the manuscript. }

\bibliographystyle{alpha}
\bibliography{bibliography}

\end{document}